\documentclass[10pt,leqno]{amsart}
\usepackage{amsmath}
\usepackage{amsfonts}
\usepackage{amssymb}
\usepackage{mathtools}
\usepackage{graphicx}
\usepackage{color}
\usepackage{hyperref}
\allowdisplaybreaks[4]

\usepackage{xcolor}
\newtheorem{theorem}{Theorem}[section]
\newtheorem*{theorem*}{Theorem}
\newtheorem{lemma}[theorem]{Lemma}

\newtheorem{proposition}[theorem]{Proposition}

\theoremstyle{definition}

\newtheorem{remark}{Remark}[section]

\newcommand{\Scal}{\operatorname{Scal}}
\newcommand{\tr}{\operatorname{tr}}
\newcommand{\R}{\mathbb{R}}
\newcommand{\KN}{\mathbin{\bigcirc\mspace{-15mu}\wedge\mspace{3mu}}}

\title[Einstein Manifolds under Two Spectral Conditions]
{Rigidity of Einstein Manifolds under Two Eigenvalue Conditions on the Curvature Operator of the Second Kind}

\author[Cheng]{Haiqing Cheng}
\address{School of Mathematical Sciences, Soochow University, Suzhou, 215006, China}
\email{chq4523@163.com}

\author[Wang]{Kui Wang}\thanks{}
\address{School of Mathematical Sciences, Soochow University, Suzhou, 215006, China}
\email{kuiwang@suda.edu.cn}


\begin{document}
\begin{abstract}
We prove a rigidity theorem for closed Einstein manifolds of dimension $n\ge6$ under two lower bounds on the curvature operator of the second kind. More precisely, for $L_1, L_2\ge0$ and some real number $\alpha>1$ in a suitable range, we consider
\[
\lambda_1\ge-L_1\bar\lambda,
\qquad
\frac1\alpha\sum_{j=1}^{\alpha}\lambda_j
\ge-L_2\bar\lambda.
\]
Here $\lambda_1\le\cdots\le\lambda_N$ are the eigenvalues of the curvature operator of the second kind $\mathring R$,
$N=(n-1)(n+2)/2$, and
$\bar\lambda=N^{-1}\sum_{j=1}^N\lambda_j$. Let $\theta(n, \alpha)$ be the constant appearing in \cite[Theorem~1.1]{CW26}; see \eqref{eq:theta-definition}. In the parameter range considered here, $L_2>\theta(n, \alpha)$, so the second condition is strictly weaker than the corresponding condition in \cite{CW26}, while the first condition and that condition do not imply each other. Under an additional explicit relation between $L_1$ and $L_2$, we prove that the manifold is either flat or a spherical space form.
\end{abstract}

\maketitle

\section{Introduction and main result}
A central problem concerning the curvature operator of the second kind is Nishikawa's conjecture \cite{Nishikawa86}. He conjectured that a closed Riemannian manifold with positive (respectively, nonnegative) curvature operator of the second kind is diffeomorphic to a spherical space form (respectively, a Riemannian locally symmetric space). The positive case was proved under the weaker assumption of two-positivity by Cao, Gursky, and Tran \cite{CGT23}. Li subsequently weakened two-positivity to three-positivity and classified the three-nonnegative case \cite{Li21}. Nienhaus, Petersen, and Wink \cite{NPW23} derived a Bochner formula for the curvature operator of the second kind and used it to prove vanishing and rigidity results under partial positivity assumptions. These results indicate that conditions on the eigenvalues of the curvature operator of the second kind provide an effective approach to problems in global Riemannian geometry.

For Einstein manifolds, Kashiwada \cite{Kashiwada93} proved that
positivity of the curvature operator of the second kind implies constant
sectional curvature. Cao, Gursky, and Tran \cite{CGT23} showed that
four-positivity implies constant sectional curvature, whereas
four-nonnegativity implies local symmetry. Li \cite{Li22JGA} extended
these results to $4\frac12$-positivity and $4\frac12$-nonnegativity,
respectively. 
Using a Bochner formula, Nienhaus, Petersen, and Wink \cite{NPW23}
proved that a compact Einstein manifold is either flat or a rational
homology sphere if its curvature operator of the second kind satisfies
a suitable partial nonnegativity condition. Dai and Fu \cite{DF24}
derived a Bochner formula for the Riemann curvature tensor on Einstein
manifolds and used it to show that a compact Einstein manifold has
constant sectional curvature under suitable nonnegativity conditions
that depend on the dimension.

More recently, Li \cite{Li25} introduced a two-parameter family of cone conditions. Let
$\lambda_1\le\cdots\le\lambda_N$ be the eigenvalues of the curvature operator of the
second kind, where $N=\frac{(n-1)(n+2)}2$ and $\bar\lambda=\frac1N\sum_{j=1}^N\lambda_j$. 
For a real number $1\le\alpha<N$, the sum of the first $\alpha$ eigenvalues is defined by
\begin{equation}\label{eq:interpolated-sum-introduction}
\sum_{j=1}^{\alpha}\lambda_j
:=
\sum_{j=1}^{\lfloor\alpha\rfloor}\lambda_j
+(\alpha-\lfloor\alpha\rfloor)\lambda_{\lfloor\alpha\rfloor+1}.
\end{equation}
For $\theta>-1$, the cone condition introduced in \cite{Li25} takes the form
\[
\frac1\alpha\sum_{j=1}^{\alpha}\lambda_j\ge-\theta\bar\lambda.
\]

For Einstein manifolds, Cheng and Wang first considered the case $\alpha=1$ and proved that a closed
Einstein manifold satisfying the corresponding lower bound is either flat or
a spherical space form \cite{CW26-1}. They later treated the case $\alpha=2$ in dimensions $n=4, 5$ and $n\ge 8$ \cite{CW25-2}. This result was extended by Fu and Lu to positive integer values
of $\alpha$ in a suitable dimension-dependent range \cite{FL25}. 
Subsequently, Cheng and Wang  extended these results to real values of
$\alpha$ \cite{CW26}. More precisely, for $n\ge6$ and $\alpha$ in the
range specified in \cite[Theorem 1.1]{CW26}, they proved that a closed
Einstein manifold satisfying
$$
\frac1\alpha\sum_{j=1}^{\alpha}\lambda_j
\ge-\theta(n, \alpha)\bar\lambda,
$$
where
\begin{equation}\label{eq:theta-definition}
\theta(n, \alpha)
=
\frac{3(N-n+1)(N-\alpha)}
{3n\alpha(N-2)+(N-3)(N-\alpha)}-1,
\end{equation}
is either flat or a spherical space form. The cases $n=4, 5$ are treated
separately in \cite[Theorem 1.2]{CW26}.



In the present paper, we replace the single cone condition by two simultaneous bounds:
one controls the smallest eigenvalue, and the other is a weaker lower bound for the first
$\alpha$ eigenvalues. Our main result is the following.

\begin{theorem}\label{thm:main}
Let $(M^n, g)$ be a connected closed Einstein manifold of dimension $n\ge6$, and set $N=\frac{(n-1)(n+2)}2$. Let
\[
1<\alpha\le
\min\left\{
\frac{n^4-n^3+8n-8}{3n^3+5n^2-22n+8},
\frac{n^2+n-8}{4n-8}
\right\},
\]
and let $\theta(n, \alpha)$ be given by \eqref{eq:theta-definition}. Suppose that
$L_1, L_2\ge0$ satisfy
\begin{equation}\label{eq:main-independence-region}
\theta(n, \alpha)<L_2<L_1<
\frac{N(\alpha-1)+\alpha(N-1)\theta(n, \alpha)}{N-\alpha}
\end{equation}
and
\begin{equation}\label{eq:main-bochner-region}
(N-3)L_2+\frac{3n(N-2)}{N-1}L_1
\le
2N-6n+6+\frac{3n}{N-1}.
\end{equation}
If the eigenvalues of $\mathring R$ satisfy, at every point of $M$,
\begin{equation}\label{eq:main-two-conditions}
\lambda_1\ge-L_1\bar\lambda,
\qquad
\frac1\alpha\sum_{j=1}^{\alpha}\lambda_j
\ge-L_2\bar\lambda,
\end{equation}
then $M$ is either flat or a spherical space form. In particular, if $M$ is simply
connected, then $M$ is isometric, up to scaling, to the round sphere.
\end{theorem}

\begin{remark}\label{rem:complementary}
Since $L_2>\theta(n, \alpha)$, the second condition in
\eqref{eq:main-two-conditions} is strictly weaker than the cone
condition considered by Cheng and Wang in \cite{CW26}, namely,
\[
\frac1\alpha\sum_{j=1}^{\alpha}\lambda_j
\ge-\theta(n, \alpha)\bar\lambda.
\]
The additional condition on the smallest eigenvalue and the cone condition considered in \cite{CW26} do not imply each other; see Lemma \ref{lem:spectral-independence}. Moreover, the pair of conditions in \eqref{eq:main-two-conditions} neither implies nor is implied by the condition in \cite{CW26}. Thus, Theorem \ref{thm:main} complements the rigidity theorem proved in \cite{CW26}.
\end{remark}

\begin{remark}
The set of parameters satisfying \eqref{eq:main-independence-region} and
\eqref{eq:main-bochner-region} is nonempty. Indeed,
$\theta(n, \alpha)$ is strictly decreasing in $\alpha$, and hence
$\theta(n, \alpha)<\theta(n, 1)$ for $\alpha>1$. A direct computation gives
\[
\frac{N(\alpha-1)+\alpha(N-1)\theta(n, \alpha)}
     {N-\alpha}
-\theta(n, \alpha)
=
\frac{N(\alpha-1)(1+\theta(n, \alpha))}{N-\alpha}>0.
\]
For $\varepsilon>0$, set
$L_2=\theta(n, \alpha)+\varepsilon$ and
$L_1=\theta(n, \alpha)+2\varepsilon$. If
\[
2\varepsilon<
\frac{N(\alpha-1)+\alpha(N-1)\theta(n, \alpha)}
     {N-\alpha}
-\theta(n, \alpha),
\]
then \eqref{eq:main-independence-region} holds. Moreover,
\[
2N-6n+6+\frac{3n}{N-1}
=
\big(
N-3+\frac{3n(N-2)}{N-1}
\big)\theta(n, 1)
>
\big(
N-3+\frac{3n(N-2)}{N-1}
\big)\theta(n, \alpha).
\]
For the above choice of $L_1$ and $L_2$, we have
\begin{align*}
&(N-3)L_2+\frac{3n(N-2)}{N-1}L_1\\
&=
\big(
N-3+\frac{3n(N-2)}{N-1}
\big)\theta(n, \alpha)
+
\varepsilon
\big(
N-3+\frac{6n(N-2)}{N-1}
\big).
\end{align*}
Therefore, by choosing $\varepsilon>0$ sufficiently small,
both \eqref{eq:main-independence-region} and
\eqref{eq:main-bochner-region} are satisfied.
\end{remark}

The paper is organized as follows. Section 2 recalls the curvature operator of the second kind, the relevant Weyl tensor identities, and the Bochner formula. In Section 3, we compare the new eigenvalue conditions with the cone condition considered in \cite{CW26}, establish the key algebraic estimate needed for the Bochner argument, and prove Theorem \ref{thm:main}.

\section{Preliminaries}
In this section, we fix the notation and recall some basic facts about the curvature operator of the second kind on Einstein manifolds. We also collect several identities and estimates involving the Weyl tensor and recall the Bochner formula for the Riemann curvature tensor. These formulas give the lower bound for $\langle\Delta R, R\rangle$ that will be used in the proof of the main theorem. For further details, we refer the reader to \cite{BK78, CGT23, CW26, Li22PAMS, Li22product, Nishikawa86}.

\subsection{The curvature operator of the second kind}

Let $(V, g)$ be an $n$-dimensional Euclidean vector space. We identify $V$ with its dual space $V^*$ via the metric $g$. Denote by $S^2(V)$ the space
of symmetric two-tensors and by
\[
S_0^2(V)=\{\varphi\in S^2(V): \tr_g\varphi=0\}
\]
its trace-free subspace. We have the orthogonal decomposition
\[
S^2(V)=S_0^2(V)\oplus\R g,
\qquad
\dim S_0^2(V)=N=\frac{(n-1)(n+2)}2.
\]
Let $R$ be an algebraic curvature tensor on $V$. It induces a self-adjoint operator
\[\bar R:S^2(V)\rightarrow S^2(V),
\qquad
(\bar R\varphi)_{ij}
=
\sum_{k,l=1}^n R_{iklj}\varphi_{kl}.\]
Let $\pi:S^2(V)\to S_0^2(V)$ be the orthogonal projection.
As observed in \cite{NPW23}, the curvature operator of the second
kind introduced by Nishikawa \cite{Nishikawa86} can be represented
by the self-adjoint operator
\[
\mathring R
:=
\pi\circ\bar R|_{S_0^2(V)}
:S_0^2(V)\rightarrow S_0^2(V).
\]
Equivalently, it is characterized by
\[\langle\mathring R\varphi,\psi\rangle
=
\langle\bar R\varphi,\psi\rangle
=
R_{iklj}\varphi_{kl}\psi_{ij},
\qquad
\varphi,\psi\in S_0^2(V).\]

When $V=T_pM$ and $R$ is the Riemann curvature tensor of an
Einstein manifold $(M^n,g)$, we denote the eigenvalues of
$\mathring R$ by
\[
\lambda_1\le\cdots\le\lambda_N.
\]
Let $\{S_j\}_{j=1}^N$ be a corresponding orthonormal eigenbasis
of $S_0^2(T_pM)$.
For an Einstein manifold, the average eigenvalue is related to the scalar curvature by
\begin{equation}\label{eq:scalar-average}
\bar\lambda=\frac1N\sum_{j=1}^N\lambda_j
=\frac{\Scal}{n(n-1)},
\end{equation}
see \cite[Proposition 2.1]{CW26-1}. For symmetric two-tensors $A$ and $B$, the Kulkarni--Nomizu product is defined by
\[
(A\KN B)_{ijkl}
=A_{ik}B_{jl}+A_{jl}B_{ik}-A_{il}B_{jk}-A_{jk}B_{il}.
\]
For an Einstein metric, the curvature decomposition is given by
\begin{equation}\label{eq:Einstein-decomposition}
R=W+\frac{\Scal}{2n(n-1)}g\KN g
=W+\frac{\bar\lambda}{2}g\KN g,
\end{equation}
where $W$ is the Weyl tensor.

\subsection{Weyl tensor identities and the Bochner formula}

If $S\in S^2(V)$ and $T$ is a covariant $k$-tensor, define
\[
(ST)(X_1, \ldots, X_k)
=
\sum_{a=1}^k
T(X_1,\ldots, SX_a, \ldots, X_k).
\]
Let $\{S_j\}_{j=1}^N$ be an orthonormal eigenbasis of $\mathring R$.
We shall use the identity
\begin{equation}\label{eq:W-eigenvalues}
|W|^2
=
\frac43
\left(
\sum_{j=1}^N\lambda_j^2-N\bar\lambda^2
\right), 
\end{equation}
see \cite{CW26, NPW23}.

The following estimate is the form of
\cite[Lemma 3.1]{CW26} used below.

\begin{lemma}\label{lem:weighted-Weyl}
Let $(M^n,g)$ be an $n$-dimensional Einstein manifold with $n\ge6$, and let $1\le\alpha\le\frac{N-3}{2(n-2)}$.
If $c\ge0$ and
\[
\frac1\alpha\sum_{j=1}^{\alpha}\lambda_j
\ge-c\bar\lambda,
\]
then
\begin{equation}\label{eq:weighted-Weyl}
\sum_{j=1}^N\lambda_j|S_jW|^2
\ge
-\frac{16(N-3)}{3n}c\bar\lambda
\left(
\sum_{j=1}^N\lambda_j^2-N\bar\lambda^2
\right).
\end{equation}
\end{lemma}

For any tensor field $Y$, we use the rough Laplacian
$\Delta Y=\tr_g(\nabla^2Y)$. In particular,
\begin{equation}\label{eq:norm-Bochner}
\Delta|R|^2
=
2|\nabla R|^2+2\langle\Delta R, R\rangle.
\end{equation}
For an Einstein manifold, the Bochner formula can be written as
\begin{align}\label{eq:exact-Bochner}
\begin{aligned}
3\langle\Delta R, R\rangle
={}&
\sum_{j=1}^N\lambda_j|S_jW|^2
-\frac{16N(2N-9n+6)}{3n}\bar\lambda^3 \\
&+
\frac{16(2N-12n+6)}{3n}\bar\lambda
\sum_{j=1}^N\lambda_j^2
+16\sum_{j=1}^N\lambda_j^3,
\end{aligned}
\end{align}
see \cite{CW26, CW26-1, DF24}. Combining
\eqref{eq:weighted-Weyl} with \eqref{eq:exact-Bochner}, we obtain
\begin{align}
3\langle\Delta R, R\rangle
\ge{}&
\frac{16N}{3n}
\bigl[(N-3)c-(2N-9n+6)\bigr]\bar\lambda^3
\nonumber\\
&+
\frac{16}{3n}
\bigl[2N-12n+6-(N-3)c\bigr]\bar\lambda
\sum_{j=1}^N\lambda_j^2
+16\sum_{j=1}^N\lambda_j^3.
\label{eq:Bochner-lower-bound}
\end{align}

For later use, write $\lambda=(\lambda_1, \ldots, \lambda_N)$ and define
\begin{equation}\label{eq:F-b}
\begin{aligned}
\mathcal F_c(\lambda)
=\frac{16}{3}\Bigg[
&\sum_{j=1}^N\lambda_j^3
+\frac{2N-12n+6-(N-3)c}{3n}\,
 \bar\lambda\sum_{j=1}^N\lambda_j^2\\
&+\frac{N\bigl[(N-3)c-(2N-9n+6)\bigr]}{3n}
 \bar\lambda^3
\Bigg].
\end{aligned}
\end{equation}
Then \eqref{eq:Bochner-lower-bound} becomes
\begin{equation}\label{eq:Delta-lower-F}
\langle\Delta R, R\rangle
\ge\mathcal F_c(\lambda).
\end{equation}
Moreover, for $c, d\ge0$,
\begin{equation}\label{eq:F-difference}
\mathcal F_c(\lambda)-\mathcal F_d(\lambda)
=
\frac{16(N-3)}{9n}(d-c)\bar\lambda
\Big(
\sum_{j=1}^N\lambda_j^2-N\bar\lambda^2
\Big).
\end{equation}

\section{Proof of the main theorem}

The following lemma shows that, since $L_2>\theta(n, \alpha)$, the bound on
the average of the first $\alpha$ eigenvalues is strictly weaker than the
condition considered in \cite{CW26}, whereas the additional smallest-eigenvalue
bound is independent of it. Hence our rigidity theorem complements the result
proved there.

\begin{lemma}\label{lem:spectral-independence}
Let $N=(n-1)(n+2)/2$, $1<\alpha\le N-1$, and let $\theta=\theta(n, \alpha)\ge0$ be defined by \eqref{eq:theta-definition}. Let $L_1, L_2\ge0$, and suppose that $\lambda_1 \le \cdots \le \lambda_N$ with $\sum_{i=1}^N\lambda_i=N\bar\lambda>0$. Consider the conditions 
\[
\mathcal A: \quad
\lambda_1\ge-L_1\bar\lambda,
\qquad
\mathcal B: \quad
\frac1\alpha\sum_{i=1}^{\alpha}\lambda_i
\ge-L_2\bar\lambda,
\]
and the condition used in \cite[Theorem 1.1]{CW26},
\[
\mathcal C:\quad
\frac1\alpha\sum_{i=1}^{\alpha}\lambda_i
\ge-\theta(n, \alpha)\bar\lambda.
\]
Then the following assertions
\begin{enumerate}
\item $\mathcal A$ and $\mathcal C$ do not imply each other;
\item $\mathcal B$ is strictly weaker than $\mathcal C$;
\item $\mathcal A$ and $\mathcal B$ do not imply each other,
\end{enumerate}
hold simultaneously if and only if
\begin{equation}\label{eq:independence-region}
\theta(n, \alpha)<L_2<L_1<
\frac{N(\alpha-1)+\alpha(N-1)\theta(n, \alpha)}
     {N-\alpha}.
\end{equation}

\end{lemma}

\begin{proof}
By homogeneity, we may assume that $\bar\lambda=1$. Write 
\[
\theta=\theta(n,\alpha),
\qquad
D(t)=\frac{N(\alpha-1)+\alpha(N-1)t}{N-\alpha}.
\]

We first compare $\mathcal A$ with $\mathcal C$. Since $\lambda_1$ is the smallest eigenvalue,
$\frac1\alpha\sum_{i=1}^{\alpha}\lambda_i\ge\lambda_1$.
Thus the condition $\mathcal A$ implies $\mathcal C$ whenever
$L_1\le\theta$. Suppose that $L_1>\theta$, and set
\[
U_t:=
\Big(
\underbrace{-t,\ldots,-t}_{\lceil\alpha\rceil},
\underbrace{
\frac{N + \lceil\alpha\rceil t}{N-\lceil\alpha\rceil},\ldots,\frac{N+ \lceil\alpha\rceil t}{N-\lceil\alpha\rceil}
}_{N-\lceil\alpha\rceil}
\Big), \quad t\ge0.
\]
The components of $U_t$ are arranged in nondecreasing order and satisfy 
\[
\sum_{i=1}^N(U_t)_i=N, \quad (U_t)_1=-t, \quad \frac1\alpha\sum_{i=1}^{\alpha}(U_t)_i=-t.
\]
Thus $U_{L_1}$ satisfies $\mathcal A$ but violates $\mathcal C$.
Therefore, the condition $\mathcal A$ implies $\mathcal C$ if and only if
$L_1\le\theta$.
We next consider the reverse implication. By
\cite[Lemma~3.2]{CW26}, the condition $\mathcal C$ gives
\[
\lambda_1\ge-D(\theta).
\]
Hence $\mathcal C$ implies $\mathcal A$ whenever
$L_1\ge D(\theta)$. To see that this condition is necessary, for
$t\ge0$ consider
\[
V_t=
\Big(
-D(t),
\underbrace{
\frac{N+\alpha t}{N-\alpha},\ldots,
\frac{N+\alpha t}{N-\alpha}
}_{N-1}
\Big).
\]
The components of $V_t$ are arranged in nondecreasing order and satisfy
\[
\sum_{i=1}^N(V_t)_i=N, \quad (V_t)_1=-D(t), \quad \frac1\alpha\sum_{i=1}^{\alpha}(V_t)_i=-t.
\]
If $L_1<D(\theta)$, then $V_\theta$ satisfies $\mathcal C$ but
violates $\mathcal A$. Consequently,
$\mathcal C$ implies $\mathcal A$ if and only if $L_1\ge D(\theta)$.
Hence, $\mathcal A$ and $\mathcal C$ do not imply each other if and
only if
\begin{equation}\label{eq:A-C-independence}
\theta<L_1<D(\theta).
\end{equation}

We now compare $\mathcal B$ with $\mathcal C$. 
If $L_2\ge\theta$, then $\mathcal C$ implies $\mathcal B$, whereas $U_\theta$ shows that this implication fails when $L_2<\theta$.
For $L_2=\theta$, the two conditions coincide, while for $L_2>\theta$, $U_{L_2}$ satisfies $\mathcal B$ but not $\mathcal C$.
Therefore, $\mathcal B$ is strictly weaker than
$\mathcal C$ if and only if
\begin{equation}\label{eq:B-C-comparison}
L_2>\theta.
\end{equation}

Finally, we compare $\mathcal A$ with $\mathcal B$. 
Since $\frac1\alpha\sum_{i=1}^{\alpha}\lambda_i\ge\lambda_1$, $\mathcal A$ implies $\mathcal B$ for $L_1\le L_2$, whereas $U_{L_1}$ shows that this implication fails for $L_1>L_2$.
Similarly, \cite[Lemma~3.2]{CW26} shows that $\mathcal B$ implies $\mathcal A$ for $L_1\ge D(L_2)$, whereas $V_{L_2}$ shows that this implication fails for $L_1<D(L_2)$.
Therefore, $\mathcal A$ and $\mathcal B$ do not imply each other if
and only if
\begin{equation}\label{eq:A-B-independence}
L_2<L_1<D(L_2).
\end{equation}

Together, conditions \eqref{eq:A-C-independence},
\eqref{eq:B-C-comparison}, and \eqref{eq:A-B-independence} are
equivalent to
\[
\theta<L_2<L_1<\min\{D(\theta),D(L_2)\}.
\]
Since $D(t)$ is strictly increasing and $L_2>\theta$, we have
$D(L_2)>D(\theta)$. Hence all three assertions hold if and only if
\[
\theta<L_2<L_1<D(\theta).
\]
Substituting the definition of $D$ yields
\eqref{eq:independence-region}.
\end{proof}

We next prove the constrained cubic estimate used in the Bochner argument.

\begin{lemma}\label{ImLm}
Let $m\ge3$ be an integer and let $x_1, \ldots, x_m\ge0$ satisfy
$\sum_{i=1}^m x_i=B_0>0$. If
\[
A_0\le A_*:=\frac{(2m-1)B_0}{m(m-1)},
\]
then
\begin{equation}\label{eq:ImLm-bound}
\sum_{i=1}^m(x_i^3-A_0x_i^2)
\ge \frac{B_0^3}{m^2}-\frac{A_0B_0^2}{m}.
\end{equation}
Equality holds if and only if either $A_0<A_*$ and
$x_1=\cdots=x_m=B_0/m$, or $A_0=A_*$ and
$(x_1, \ldots, x_m)$ is a permutation of
\[
\left(\frac{B_0}{m}, \ldots, \frac{B_0}{m}\right)
\quad \text{or} \quad
\left(0, \frac{B_0}{m-1}, \ldots, \frac{B_0}{m-1}\right).
\]
\end{lemma}

\begin{proof}
The proof follows the same strategy as that of \cite[Proposition 3.4]{CW26-1}. For the reader's convenience, we provide the full details below.
Define
\[
f(x)=\sum_{i=1}^m(x_i^3-A_0x_i^2),
\qquad
\Sigma=
\Big\{x\in\mathbb{R}^m: 
x_i\ge 0, \ \sum_{i=1}^m x_i=B_0\Big\}.
\]
As $\Sigma$ is compact and $f$ is continuous, $f$ attains
its global minimum on $\Sigma$.
Let the corresponding Lagrangian function be
\[
\mathcal{L}(x, \mu)=\sum_{i=1}^m(x_i^3-A_0x_i^2)
+ \mu(\sum_{i=1}^m x_i - B_0),
\]
and let
\begin{align*}
\Omega_k = \left\{ x \in \Sigma :  x_1 = \cdots = x_{m-k} = 0, \, x_{m-k+1}, \ldots, x_m > 0 \right\}
\end{align*}
for $k = 1, \ldots, m$, where the condition on zero entries is omitted when $k=m$.
In $\Omega_k$, using the method of Lagrange multipliers, we find that the critical points
\[x=(\underbrace{0, \cdots, 0}_{m-k}, x_{m-k+1}, \cdots, x_m)\]
of $\mathcal{L}(x, \mu)$ satisfy
\[
\begin{cases}
3x_i^2-2A_0x_i+\mu=0, 
& i=m-k+1, \ldots, m,\\
\sum_{i=m-k+1}^{m}x_i=B_0.
\end{cases}
\]
Hence the possible critical points are given by
\begin{equation*}
P_{k,l}:= (\underbrace{0, \ldots, 0}_{m-k}, \underbrace{a, \ldots, a}_{l}, \underbrace{b, \ldots, b}_{k-l}),
\end{equation*}
for $0 \le l \le \frac{k}{2}$, where $a$ and $b$ are defined through
\begin{equation}\label{lm1:eqns1}
\begin{cases}
a + b = \frac{2A_0}{3}:=2A, \\
l a + (k-l)b = B_0.
\end{cases}
\end{equation}
If $l=\frac k2$ and \eqref{lm1:eqns1} admits a solution, then $a+b = \frac{2A_0}{3} = \frac{2B_0}{k}$, which yields $A_0=\frac{3B_0}{k}$. However, since $m\ge3$ and $k\le m$, we have
\[
A_0\le A_*<\frac{3B_0}{m}\le\frac{3B_0}{k},
\]
which contradicts $A_0=\frac{3B_0}{k}$. 
If $0 \le l < \frac{k}{2}$, according to \eqref{lm1:eqns1}, we have
\[
a = A - \frac{B_0 - kA}{k - 2l}, \qquad b = A + \frac{B_0 - kA}{k - 2l}.
\]
One can verify that
\begin{align*}
f(P_{k, l}) 
&= -2kA^3 - 3A^2 (B_0 - kA) + \frac{(B_0 - kA)^3}{(k-2l)^2}.
\end{align*}
Note that 
\[
A_0\le A_*<\frac{3B_0}{m}\le\frac{3B_0}{k},
\]
and consequently, $B_0-kA>0$. 
For $1\le l<k/2$, we have
\[
f(P_{k, l})-f(P_{k,0})
=(B_0-kA)^3
\left(
\frac{1}{(k-2l)^2}-\frac{1}{k^2}
\right)>0.
\]
Hence every global minimizer is a permutation of $P_{k,0}$
for some $1\le k\le m$.

It remains to compare the values of $f$ at $P_{k, 0}$ with $1\le k\le m$.
For $1\le k<m$, direct calculation gives
\begin{align}\label{lm:eqn-comparison}
\begin{aligned}
f(P_{k, 0})-f(P_{m,0})
&=\frac{B_0^3}{k^2}-\frac{A_0B_0^2}{k} - \left(\frac{B_0^3}{m^2} - \frac{A_0B_0^2}{m}\right)\\
&=B_0^2\left(\frac{1}{k}-\frac{1}{m}\right)
\left[
B_0\left(\frac{1}{k}+\frac{1}{m}\right)-A_0
\right].
\end{aligned}
\end{align}
Since $k\le m-1$,
\[
B_0\left(\frac{1}{k}+\frac{1}{m}\right)
\ge
B_0\left(\frac{1}{m-1}+\frac{1}{m}\right)
=A_*.
\]
Together with $A_0\le A_*$, this implies $f(P_{k, 0})\ge f(P_{m, 0})$.
Hence
\[
\min_{x\in\Sigma} f(x)
=f(P_{m, 0})
=\frac{B_0^3}{m^2}-\frac{A_0B_0^2}{m},
\]
which proves \eqref{eq:ImLm-bound}.

If $A_0<A_*$, then \eqref{lm:eqn-comparison} yields $f(P_{k, 0}) > f(P_{m, 0})$ for every $k<m$.
Thus $P_{m, 0}$ is the unique global minimizer.
If $A_0=A_*$, then for $k<m$ equality in
\eqref{lm:eqn-comparison} holds if and only if $k=m-1$.
Therefore the global minimizers are precisely $P_{m, 0}$ and the permutations of $P_{m-1, 0}$, proving the stated equality cases.
\end{proof}

We now establish the pointwise Bochner estimate under the two new spectral conditions.
In the following proposition, we do not assume the independence inequalities from Lemma \ref{lem:spectral-independence}.

\begin{proposition}\label{prop:two-spectral-conditions}
Let $(M^n, g)$ be an Einstein manifold with $n\ge6$. Set $N=\frac{(n-1)(n+2)}2$, and assume $1<\alpha\le\frac{N-3}{2(n-2)}$.
Suppose that
\begin{equation}\label{eq:two-spectral-conditions}
\lambda_1\ge-L_1\bar\lambda,
\qquad
\frac1\alpha\sum_{j=1}^{\alpha}\lambda_j
\ge-L_2\bar\lambda,
\end{equation}
where $L_1, L_2\ge0$ satisfy
\begin{equation}\label{eq:parameter-relation}
(N-3)L_2+\frac{3n(N-2)}{N-1}L_1
\le
2N-6n+6+\frac{3n}{N-1}.
\end{equation}
Then
\[
\langle\Delta R, R\rangle\ge0.
\]
Regarding the equality case, if \eqref{eq:parameter-relation} is
strict, then $\langle\Delta R, R\rangle=0$
if and only if
\begin{equation}\label{eq:equal-spectrum}
\lambda_1=\cdots=\lambda_N=\bar\lambda.
\end{equation}
If \eqref{eq:parameter-relation} is an equality and
$\langle\Delta R, R\rangle=0$ at a point, then either
\eqref{eq:equal-spectrum} holds or
\begin{equation}\label{eq:exceptional-spectrum}
(\lambda_1, \ldots, \lambda_N)
=
\left(
-L_1,
\frac{N+L_1}{N-1},\ldots,
\frac{N+L_1}{N-1}
\right)\bar\lambda.
\end{equation}
\end{proposition}

\begin{proof}
We work at a fixed point. The first part of \eqref{eq:two-spectral-conditions} shows
\[
-L_1\bar\lambda
\le\lambda_1
\le\frac1N\sum_{j=1}^N\lambda_j
=\bar\lambda,
\]
so we have $\bar\lambda\ge0$. If $\bar\lambda=0$, then
\eqref{eq:two-spectral-conditions} and
$\sum_{j=1}^N\lambda_j=0$ imply that all $\lambda_j$ vanish.
It follows from \eqref{eq:W-eigenvalues} and
\eqref{eq:Einstein-decomposition} that $W=0$ and $R=0$.

Suppose now that $\bar\lambda>0$. Applying
\eqref{eq:Delta-lower-F} with $c=L_2$, we obtain
\begin{equation}\label{eq:Delta-lower-F-L2}
\langle\Delta R, R\rangle
\ge\mathcal F_{L_2}(\lambda),
\end{equation}
where 
\begin{equation*}
\begin{aligned}
\mathcal F_{L_2}(\lambda)
=\frac{16 \bar\lambda^3}{3}\Big[
&\sum_{j=1}^N(\frac{\lambda_j}{\bar\lambda})^3
+\frac{2N-12n+6-(N-3)L_2}{3n}\,
\sum_{j=1}^N(\frac{\lambda_j}{\bar\lambda})^2\\
&-\frac{N\bigl[(2N-9n+6)-(N-3)L_2\bigr]}{3n}
\Big].
\end{aligned}
\end{equation*}
Setting  $x_j:={\lambda_j}/{\bar\lambda}+L_1$ gives $x_j\ge0$ and $\sum_{j=1}^Nx_j=N(1+L_1)$.
The second condition in \eqref{eq:two-spectral-conditions} also gives
\[
\frac1\alpha\sum_{j=1}^{\alpha}x_j\ge L_1-L_2.
\]
In what follows, we omit this additional constraint from the minimization problem, since this only enlarges the admissible set. Therefore, the lower bound obtained from Lemma~\ref{ImLm} remains valid on the original admissible set.
Set
\[
A_0:=3L_1-
\frac{2N-12n+6-(N-3)L_2}{3n}.
\]
A direct computation shows
\begin{equation}\label{eq:shifted-objective}
\begin{aligned}
\frac{3 \mathcal F_{L_2}(\lambda)}{16 \bar\lambda^3}
&=\sum_{j=1}^N(x_j-L_1)^3
+(3L_1-A_0)\,
\sum_{j=1}^N (x_j-L_1)^2 - N(3L_1-A_0+1)\\
&= \sum_{j=1}^N x_j^3 - A_0 \sum_{j=1}^N x_j^2
+N(1+L_1)^2\bigl(A_0-1-L_1\bigr).
\end{aligned}
\end{equation}
Here we used
\begin{align*}
&3L_1^2 N(1+L_1)-NL_1^3 + (3L_1 - A_0) (-2L_1 N(1+L_1)+NL_1^2) \\
&\quad - N(3L_1-A_0+1)= N(1+L_1)^2 (A_0-1-L_1).
\end{align*}
It can be verified that 
\begin{align*}
&\frac{(2N-1)N(1+L_1)}{N(N-1)}-A_0\\
&\quad=
\frac1{3n}
\left(
2N-6n+6+\frac{3n}{N-1}
-(N-3)L_2
-\frac{3n(N-2)}{N-1}L_1
\right)
\ge0.
\end{align*}
Applying Lemma \ref{ImLm} with $m=N,\ B_0=N(1+L_1),$ we get
\[
\sum_{j=1}^N(x_j^3-A_0x_j^2)
\ge
N(1+L_1)^3-A_0N(1+L_1)^2.
\]
Then \eqref{eq:Delta-lower-F-L2} and  \eqref{eq:shifted-objective} yield
\[
\langle\Delta R, R\rangle \ge \mathcal F_{L_2}(\lambda)\ge0.
\]

We now consider the equality cases. If
$\langle\Delta R, R\rangle=0$, then
\[
0=\langle\Delta R, R\rangle
\ge\mathcal F_{L_2}(\lambda)\ge0.
\]
Hence $\mathcal F_{L_2}(\lambda)=0$, so equality holds in the
application of Lemma \ref{ImLm}.
If \eqref{eq:parameter-relation} is strict, then $A_0<A_*$, and the
equality case of Lemma \ref{ImLm} gives
\[
x_1=\cdots=x_N=1+L_1,
\]
which is equivalent to \eqref{eq:equal-spectrum}.
If \eqref{eq:parameter-relation} is an equality, then $A_0=A_*$.
Lemma \ref{ImLm} gives either 
the same point with all coordinates equal or, up to
permutation,
\[
(x_1, \ldots, x_N)
=
\left(
0,
\frac{N(1+L_1)}{N-1},\ldots,
\frac{N(1+L_1)}{N-1}
\right).
\]
Since $\lambda_1\le\cdots\le\lambda_N$, we also have
$x_1\le\cdots\le x_N$, so the zero entry must be $x_1$. Returning to
the variables $\lambda_j$ gives \eqref{eq:exceptional-spectrum}.
Finally, if \eqref{eq:equal-spectrum} holds, then
\eqref{eq:W-eigenvalues} gives $W=0$, and
\eqref{eq:exact-Bochner} gives
$\langle\Delta R, R\rangle=0$.
\end{proof}

\begin{proof}[Proof of Theorem \ref{thm:main}]
By \eqref{eq:main-bochner-region} and
Proposition \ref{prop:two-spectral-conditions}, we have
\[
\Delta|R|^2
=
2|\nabla R|^2+2\langle\Delta R, R\rangle
\ge 2|\nabla R|^2
\]
everywhere on $M$. 
Since $M$ is closed, the divergence theorem gives
\[
0
=
\int_M\Delta|R|^2\,d\mu_g
=
2\int_M|\nabla R|^2\,d\mu_g
+
2\int_M\langle\Delta R, R\rangle\,d\mu_g.
\]
Both integrands are nonnegative, and hence
\begin{equation}\label{eq:global-equality}
\nabla R=0,
\qquad
\langle\Delta R, R\rangle=0
\end{equation}
everywhere on $M$.

As observed in the proof of
Proposition \ref{prop:two-spectral-conditions}, we have
$\bar\lambda\ge0$, and $\bar\lambda=0$ implies $R=0$.
It therefore remains to consider the case $\bar\lambda>0$.
By the equality statement in Proposition \ref{prop:two-spectral-conditions}, at each point either all eigenvalues satisfy $\lambda_1=\cdots=\lambda_N=\bar\lambda$ as in \eqref{eq:equal-spectrum}, or, when \eqref{eq:main-bochner-region} is an equality, they are given by \eqref{eq:exceptional-spectrum}. For the eigenvalues in \eqref{eq:exceptional-spectrum}, direct calculation gives 
\begin{equation}\label{eq:exceptional-alpha-average}
\frac1\alpha\sum_{j=1}^{\alpha}\lambda_j
=
\frac{N(\alpha-1)-(N-\alpha)L_1}
     {\alpha(N-1)}\bar\lambda
>
-\theta(n, \alpha)\bar\lambda,
\end{equation}
where the strict inequality follows from \eqref{eq:main-independence-region}. 
Therefore, applying \eqref{eq:Delta-lower-F} with 
$c=\theta(n, \alpha)$, we obtain
\[
\langle\Delta R, R\rangle
\ge\mathcal F_{\theta(n, \alpha)}(\lambda).
\]
On the other hand, the proof of
Proposition \ref{prop:two-spectral-conditions} gives
$\mathcal F_{L_2}(\lambda)\ge0$. Thus
\eqref{eq:global-equality} implies
\[
0=\langle\Delta R, R\rangle
\ge\mathcal F_{L_2}(\lambda)\ge0,
\]
and consequently $\mathcal F_{L_2}(\lambda)=0$.
For brevity, write $\theta=\theta(n, \alpha)$. Using
\eqref{eq:F-difference} and \eqref{eq:exceptional-spectrum}, we obtain
\begin{align*}
\mathcal F_\theta(\lambda)
= \mathcal F_\theta(\lambda)-\mathcal F_{L_2}(\lambda)
&=
\frac{16(N-3)}{9n}(L_2-\theta)\bar\lambda
\left(
\sum_{j=1}^N\lambda_j^2-N\bar\lambda^2
\right)\\
&= \frac{16(N-3)}{9n}(L_2-\theta)
\frac{N(1+L_1)^2}{N-1}\bar\lambda^3
>0,
\end{align*}
where the last inequality follows from
$L_2>\theta$ in \eqref{eq:main-independence-region} and
$\bar\lambda>0$. Hence
\[
\langle\Delta R, R\rangle
\ge\mathcal F_\theta(\lambda)>0,
\]
contradicting \eqref{eq:global-equality}. This contradiction rules out the second equality case.

Consequently,
\[
\lambda_1=\cdots=\lambda_N=\bar\lambda
\]
at every point. Equation \eqref{eq:W-eigenvalues} gives $W=0$, and
\eqref{eq:Einstein-decomposition} reduces to
\[
R=\frac{\bar\lambda}{2}g\KN g.
\]
Since the scalar curvature of an Einstein manifold is constant,
$M$ has constant sectional curvature $\bar\lambda$. If
$\bar\lambda=0$, then $M$ is flat. If $\bar\lambda>0$, its universal
cover is a round sphere, and hence $M$ is a spherical space form.
If $M$ is simply connected, then $M$ is isometric, up to scaling, to the round sphere.
\end{proof}

\section*{Acknowledgments}
The first author expresses her gratitude to her advisor, Professor Ying Zhang, for lots of encouragement and helpful suggestions. 
The authors would like to thank Professor Xiaolong Li for suggesting the problem studied in this paper and for many helpful discussions.
The research of this paper is partially supported by  NSF of Jiangsu Province Grant No.BK20231309.

\bibliographystyle{plain}
\bibliography{ref}

\end{document}